\documentclass[12pt, a4paper]{article}

\usepackage[latin1]{inputenc}
\usepackage[english]{babel}
\usepackage{amssymb}
\usepackage{amsfonts}
\usepackage{amsmath}
\usepackage{amsthm}
\usepackage{epsfig}
\usepackage{graphics}
\usepackage{dsfont}
\usepackage[usenames, dvipsnames]{xcolor}
\usepackage{verbatim}
\usepackage{latexsym}
\usepackage{setspace}

\theoremstyle{plain}
\newtheorem{theorem}{Theorem}[section]
\newtheorem{corollary}[theorem]{Corollary}

\newtheorem{proposition}[theorem]{Proposition}

\theoremstyle{definition}
\newtheorem{definition}[theorem]{Definition}

\theoremstyle{definition}
\newtheorem{remark}[theorem]{Remark}

\newtheorem{example}[theorem]{Example}  
\newtheorem{examples}[theorem]{Examples} 
\allowdisplaybreaks[2]

\newcommand{\RR}{\mathbb{R}}

\newcommand{\EE}{\mathbb{E}}
\newcommand{\cB}{\mathcal{B}}
\newcommand{\cS}{\mathcal{S}}

\newcommand{\cA}{\mathcal{A}}

\newcommand{\cL}{\mathcal{L}}

\newcommand{\dd}{\mathrm{d}}
\newcommand{\ee}{\mathrm{e}}
\newcommand{\sign}{\mathrm{sgn}}

\newcommand{\abs}[1]{\left| #1 \right|}
\newcommand{\norm}[1]{\left\| #1 \right\|}

\newcommand{\bbr}{\mathbb{R}}

\newcommand{\bbp}{\mathbb{P}}
\newcommand{\bbe}{\mathbb{E}}

\newcommand{\alex}{\textcolor{black}}

\begin{document}
\title{Laplace Symbols and Invariant Distributions} 
\author{Anita Behme\thanks{Technische Universit\"at
      Dresden, Institut f\"ur Mathematische Stochastik, Zellescher Weg 12-14, 01069 Dresden, Germany, \texttt{anita.behme@tu-dresden.de}, phone: +49-351-463-32426, fax:  +49-351-463-37251.}$\;$
 and Alexander
  Schnurr\thanks{Universit\"at Siegen, Department Mathematik, Walter-Flex-Str. 3, D-57072 Siegen, Germany,
              \texttt{schnurr@mathematik.uni-siegen.de}, 
              phone: +49-271-740-3806, fax: +49-271-740-3627.
              }}
\date{\today}
\maketitle

\vspace{-1cm}
\begin{abstract}
We introduce a new kind of symbol in the framework of It\^o processes which are bounded on one side. The connection between this symbol and the infinitesimal generator is analyzed. Based on this concept, an integral criterion for invariant distributions of the underlying process is derived. Some applications are mentioned. 
\end{abstract}



\vspace*{-7mm}
\section{Introduction}\label{S0}
\setcounter{equation}{0}

Over the last two decades, the so-called symbol of a stochastic process has proven to be a useful tool in order to derive local and global properties of the corresponding stochastic process (cf. \cite{walterhabil}, \cite{schilling98hdd}, \cite{alexgeneralsymbol} and Chapter 5 of \cite{ReneLM13}). Following ideas of Jacob \cite{Jaco1998} and Schilling \cite{schilling98pos}, Schnurr has generalized this concept to homogeneous diffusions with jumps in the sense of Jacod and Shiryaev \cite{jacodshir}. In its most general form the symbol of an $\RR^d$-valued homogeneous diffusion with jumps $(X_t)_{t\geq 0}$ looks as follows. For $x,\xi\in\bbr^d$
\begin{align} \label{eq:symbol}
     p(x,\xi):=- \lim_{t\downarrow 0}\frac{\bbe^x \ee^{i(X^\sigma_t-x)'\xi}-1}{t},
\end{align}
where $X^\sigma$ is the process $X$ stopped at time $\sigma$, the first exit time of a compact neighborhood of $x$, and $x'$ denotes the transpose of the vector $x$. 
Under mild regularity conditions the limit in \eqref{eq:symbol} is always a continuous negative definite function in the sense of Schoenberg (cf. \cite{bergforst}), which means that for each fixed $x$ we obtain a L\'evy-Khintchine exponent
\begin{align} \label{lkf}
\phi(\xi):=-i\ell' \xi + \frac12 \xi' Q \xi - \int_{\RR^d} \left(\ee^{i\xi'y} -1 - i\xi' y \mathds{1}_{\{\|y\|<1\}}(y)\right) N(\dd y),\, \xi\in\RR^d.
\end{align}

In \cite{behmeschnurr1} we derived an integral criterion for invariant measures of It\^o processes based on the symbol. In that article, we had to use the above formula without the stopping time as obviously global properties might be destroyed by stopping, in contrast to local path properties. This resulted in the fact that for the method proposed in \cite{behmeschnurr1} the symbol had to satisfy a certain growth condition corresponding to bounded semimartingale characteristics of the treated processes and hence to bounded coefficients of the SDEs in the background. 
Since this is a serious restriction, in the present paper we aim to show an alternative to the classical symbol, which then can be used in cases without bounded characteristics/coefficients as long as the corresponding processes have a lower or upper bound (e.g. are restricted to be positive in each component). Hereby, we take up the classic idea of using Laplace transforms instead of Fourier transforms/characteristic functions and define a so called Laplace symbol. In the present paper we emphasize the applicability of the new concept by considering stationary distributions. However, we believe that it can be beneficial in other areas as well.

Thus, after setting the stage in Section~\ref{S1}, we will define the Laplace symbol in Section~\ref{S2} which also includes several results on its computation as well as on its relation to the generator of the stochastic process. Section~\ref{S3} is devoted to  the derivation of an integral criterion for invariant measures of the underlying processes which relies on the Laplace symbol. Some examples are added to exhibit the usability of the derived criterion. 


\section{Preliminaries}\label{S1}
\setcounter{equation}{0}


In this paper we will focus on the class of It\^o processes as defined below. 

\begin{definition} \label{def:ito}
An \emph{It\^o process} $(X_t)_{t\geq 0}$ is a strong Markov process, which is a semimartingale with respect to every $\bbp^x$ having semimartingale characteristics of the form 
\begin{align} \begin{split} \label{chars}
  B_t^{(j)}(\omega) &=\int_0^t  \ell^{(j)}(X_s(\omega)) \ \dd s,  \hspace{10mm} j=1,...,d,\\
  C_t^{jk}(\omega)  &=\int_0^t Q^{jk}(X_s(\omega)) \ \dd s,       \hspace{10mm} j,k=1,...,d,\\
  \nu(\omega;\dd s,\dd y) &=N(X_s(\omega),\dd y) \ \dd s, 
\end{split} \end{align}
for every $x\in\bbr^d$ with respect to a fixed cut-off function $\chi$. Here $\ell(x)=(\ell^{(1)}(x),...,\ell^{(d)}(x))'$ is a vector in $\bbr^d$, $Q(x)$ is a positive semi-definite matrix and $N$ is a Borel transition kernel such that $N(x,\{0\})=0$. We call $\ell$, $Q$ and $n:=\int_{y\neq 0} (1\wedge \norm{y}^2) \ N(\cdot,\dd y)$ the \emph{differential characteristics} of the process.
\end{definition}

It\^o processes have been characterized in \cite{cinlarjacod81} as the set of solutions of very general stochastic differential equations (SDEs) of Skorokhod-type. In particular this class includes L\'evy processes, solutions of L\'evy driven SDEs and Feller processes with sufficiently rich domain. 
Note that It\^o processes in the sense of Definition \ref{def:ito} are sometimes in the literature called {\em L\'evy type processes}  (cf. \cite{ReneLM13}).

Continuity of the differential characteristics of the treated It\^o processes is always sufficient for our purposes. However, we use the concept of fine continuity in order to derive even more general results. Loosely speaking, the advantage of fine continuity is that one has to care about continuity only as far as the process can detect it.  
We will use fine continuity in a way that is governed by the subsequent result which was established in \cite[Thm. II.4.8]{blumenthalget}.

\begin{proposition}
Let $X$ be a Markov process and $f:\bbr^d\to\bbr$ be a Borel-measurable function. Then $f$ is finely continuous if and only if the function
$
t\mapsto f(X_t)=f\circ X_t
$
is right continuous at zero $\bbp^x$-a.s. for every $x\in\bbr^d$.
\end{proposition}

This proposition yields another advantage of fine continuity: Consider a process $(X_t)_{t\geq 0}$ which is defined on any measurable subset $B$ of $\bbr^d$ or $\bbr_+^d$. If a function is finely continuous on the state space $B$, then we can always extend the process and the function to $\bbr^d$ or $\bbr_+^d$ by setting $X_t=x$ for $x\notin B$ and $t\geq 0$. The resulting function is again finely continuous.


\section{The Laplace symbol}\label{S2}
\setcounter{equation}{0}


\subsection{Definition of the Laplace symbol}


The Laplace symbol can be seen as a state-space dependent right hand side derivative at zero of the Laplace transform of a stochastic process. Since the Laplace transform characterizes the distribution of a random variable, the Laplace symbol in a certain way reflects the infinitesimal changes in the distribution over time. 

\begin{definition} \label{def:laplacesymbol}
Let $(X_t)_{t\geq 0}$ be a Markov process in $\RR^d_+$. Define for every $x,\xi\in\bbr_+^d$ 
\begin{align} \label{symbollimit}
\lambda(x,\xi):=-\lim_{t\downarrow 0} h_\xi(x,t):=-\lim_{t\downarrow 0}\frac{\bbe^x \ee^{-(X_t-x)'\xi}-1}{t}.
\end{align}
Then we call $\lambda:\bbr^d_+\times \bbr^d_+\to \RR$ the \emph{Laplace symbol} of $X$ with domain $D_\lambda \subseteq \bbr^d_+\times \bbr^d_+$ whenever the limit in \eqref{symbollimit} exists for every $x,\xi\in D_\lambda$.
\end{definition}

\begin{remark}
 Obviously the Laplace symbol as in Definition \ref{def:laplacesymbol} can also be defined for any Markov process which is componentwise bounded from below. Throughout this article we use the bound $0$ to ease notation. Similarly, for Markov processes bounded from above, one may define a one-sided symbol $\lambda_-:\bbr^d_-\times \bbr^d_-\to \RR$.
\end{remark}


A class of processes for which the Laplace symbol is obviously relevant are It\^o processes conditioned to stay positive. As examples we consider subordinators and a Brownian motion absorbed in zero.

\begin{examples} \begin{enumerate}
		\item A subordinator is a L\'evy process in $\RR_+$ with only positive increments and it is well known (cf. \cite{sato}) that the Laplace transform of such a process $(X_t)_{t\geq 0}$ can be written as
		$
		\bbe^0\left[\ee^{-\xi X_t} \right]=\ee^{- t \lambda(\xi)}
		$
		where
		\[
		\lambda(\xi)=\ell\xi-\int_{(0,\infty)} \left(\ee^{-\xi y}-1\right) \ N(\dd y),
		\]
		with $\ell \geq 0$ and a L\'evy measure $N$ on $\RR_+$ such that $\int_{(0,\infty)} (1 \wedge x ) N(\dd x)<\infty$. Using the fact that subordinators are homogeneous in space we obtain
		\begin{align*}
		-\frac{\bbe^x \left[ \ee^{-(X_t-x)\xi}\right] -1}{t}&= -\frac{\bbe^0 \left[ \ee^{-X_t\xi}\right] -1}{t} = -\frac{  \ee^{-t \lambda(\xi)} -1}{t}\overset{t\downarrow 0}\longrightarrow \lambda(\xi).
		\end{align*}
		In this case the Laplace symbol is constant in the first variable and coincides with the Laplace exponent.
		\item  Let $(B_t)_{t\geq 0}$ be a standard Brownian motion in $\RR$ and define $(X_t)_{t\geq 0}$ via $X_t^x=x+B_{t\wedge \tau}$, where $\tau=\inf\{t\geq 0, B_t=-x\}$. Then it follows by standard computations that
		$$\lambda(x,\xi)=- \frac{d}{dt} \bbe^x [\ee^{-(X_t-x)\xi}]|_{t=0} =- \frac{d}{dt} \EE^0 [\ee^{- B_{t\wedge \tau}\xi}]|_{t=0} = -\frac{\xi^2}{2},  \quad x> 0, \xi \in \RR_+,$$
		while $$\lambda(0,\xi)=- \frac{d}{dt} \bbe^0 [\ee^{-B_\tau \xi}]|_{t=0}= 0,  \quad \xi \in \RR_+.$$
	\end{enumerate}
\end{examples}

\begin{remark}
Obviously, by comparing the symbol and the Laplace symbol defined above, one recognizes that at least formally 
$ \text{`}\lambda(x, \xi)=p(x, i\xi)\text{'}$.
Nevertheless, we chose not to define the Laplace symbol via an analytic extension of the symbol, since for a characteristic function to be analytic it is necessary that the corresponding distribution has moments of all orders (see \cite[p. 198]{lukacs}). As we are interested in an extension of the results in \cite{behmeschnurr1} to the case of unbounded characteristics/moments, this restriction obviously would have been too strong.
\end{remark}

\vspace*{-7mm}

\subsection{Computing the Laplace symbol}

In order to establish the following main result of this section, we need a boundedness assumption which is very weak. In fact it is sufficient that for each differential characteristic $d$ the function $d(\cdot) \exp(-\cdot)$ is bounded. Hence, polynomially bounded is sufficient in order to establish the following result and for all our applications. 

\begin{theorem} \label{glmkonv}  
Let $(X_t)_{t\geq 0}$ be an $\RR_+^d$-valued It\^o process with polynomially bounded, finely continuous differential characteristics.
For every $\xi\in\bbr^d$ the Laplace symbol 
\begin{align*} 
    \lambda(x,\xi)=-\lim_{t\downarrow 0} h_\xi(x,t)=-\lim_{t\downarrow 0}\frac{\bbe^x \ee^{-(X_t-x)'\xi}-1}{t}
\end{align*}
exists and the functions $h_\xi(x,\cdot)$ are globally bounded in $t$ for every $x, \xi\in\bbr_+^d$. Furthermore, $\exp(-x'\xi) h_\xi(x,t)$ is globally bounded in $x$ and $t$ for each $\xi$. As limit we obtain
\begin{align}\label{symbol}
  \lambda(x,\xi)=\ell(x)'\xi - \frac{1}{2} \xi'Q(x) \xi -\int_{\RR^d\backslash\{0\}} \Big(\ee^{-y'\xi}-1 + y'\xi\cdot\chi(y)\Big) \ N(x,\dd y).
\end{align}

\end{theorem}

\begin{proof}
We consider the one dimensional situation, since the multidimensional version works alike. Since the proof is very similar to the one of \cite[Lemma 3.4]{behmeschnurr1} we only sketch it here. Let $x, \xi \in\bbr_+$. First use It\^o's formula under the expectation and obtain
\begin{align}
\frac{1}{t} \bbe^x \left( \ee^{-(X_t-x)\xi } -1 \right)
&= \frac{1}{t} \bbe^x \left(\int_{(0,t]} - \xi \ee^{-(X_{s-}-x)\xi} \ \dd X_s \right) \tag{I}\label{termone}\\
&\quad + \frac{1}{t} \bbe^x\left(\frac{1}{2} \int_{(0,t]} \xi^2 \ee^{-(X_{s-}-x)\xi} \ \dd [X,X]_s^c \right) \tag{II} \label{termtwo}\\
&\quad + \frac{1}{t} \bbe^x\left(\ee^{x\xi} \sum_{0<s\leq t} \Big(\ee^{- \xi X_s}-\ee^{-\xi X_{s-}}+\xi \ee^{-\xi X_{s-}} \Delta X_s \Big)\right). \tag{III}\label{termthree}
\end{align}
One has to deal with these terms one-by-one. 

In order to calculate term (I) we use the canonical decomposition of a semimartingale (see \cite[Thm. II.2.34]{jacodshir}).
It is easy to show that the integrals with respect to the martingale parts yield again martingales.
Hence, the expected values of these two parts are zero. The remaining jump part has to be put together with term (III). 

In term \eqref{termtwo} we have $[X,X]_t^c= [X^c, X^c]_t=C_t =\int_{0+}^t Q(X_{s}) \dd s$. Hence
\begin{align} \label{liml}
\frac{1}{2} \int_{(0,t]} \xi^2 \ee^{-(X_{s-}-x)\xi} \ \dd[X,X]_s^c
&=  \frac{1}{2} \xi^2 \int_{(0,t]}  \ee^{-(X_{s-}-x)\xi}  Q(X_{s}) \ \dd s.
\end{align}
Since $Q$ is finely continuous and since $w\mapsto\ee^{-w\xi}  Q(w)$ is bounded we obtain by dominated convergence
\[
\lim_{t\downarrow 0} \frac{1}{2}\xi^2 \frac{1}{t} \bbe^x \int_{(0,t]} \ee^{ -(X_{s}-x)\xi} Q(X_{s}) \ \dd s =\frac{1}{2}\xi^2 Q(x).
\]
The remaining part of term (I) as well as the jump part work alike.
Putting the terms together, we obtain in addition
\begin{align*}
\abs{\frac{\bbe^x \ee^{-(X_t-x)'\xi}-1}{t}}&=
  \left|-\xi\frac{1}{t} \alex{e^{x'\xi}} \bbe^x \int_{(0,t]} \ee^{- X_{s-}\xi} \ell(X_{s})  \ \dd s + \frac{1}{2} \xi^2\frac{1}{t} \bbe^x  \int_{(0,t]}  \ee^{- X_{s-}\xi}  Q(X_{s}) \ \dd s  \right.\\ 
& \quad\quad  + \left.\frac{1}{t} \bbe^x \int_{(0,t]} \ee^{- X_{s-}\xi} \int_{\RR\backslash \{0\}} \Big(\ee^{-y \xi}-1+y \xi \chi(y)\Big) \ N(X_{s},\dd y) \ \dd s \right| \\
&\leq\abs{\xi} e^{x'\xi} \frac{t}{t}  \norm{\ee^{-\xi \cdot}\ell(\cdot)}_\infty+ \xi^2 e^{x'\xi}\frac{t}{2t}  \norm{\ee^{-\xi\cdot} Q(\cdot)}_\infty \\ 
& \quad \quad + C_\xi e^{x'\xi}\frac{t}{t}  \norm{ \ee^{-\xi \cdot} \int_{y\neq 0} (1\wedge \abs{y}^2) \ N(\cdot,\dd y) }_\infty.
\end{align*}
\alex{Thus we have obtained the desired bounds. }
\end{proof}

As a consequence of the above result it is a simple task to derive the semimartingale characteristics of the process from the Laplace symbol, if they have not been known a-priori. 
In many cases It\^o processes are described by a L\'evy driven SDE. If this is the case, one can directly determine the Laplace symbol from the SDE and the characteristic exponent of the driving L\'evy process as shown in the following theorem. We will write $X^x$ for the solution starting in $x$. Recall that it is always possible to enlarge the probability space and define a family of probability measures $(\bbp^x)_{x\in\bbr_+^d}$ on this enlargement in a way that $\bbp(X_t^x\in B) = \bbp^x(X_t\in B)$ (cf. \cite[Section 5.6]{protter}). 

\begin{theorem} \label{thm:laplacesymbolausSDE}
 Let $(L_t)_{t\geq 0}$ be an $\RR^n$-valued L\'evy process with characteristic exponent $\phi_L(\xi):=-\log \EE[\ee^{iL_1'\xi}]$ given by \eqref{lkf} 
and consider the SDE
\begin{equation} \label{eq-sde}
 \dd X^x_t=\Phi(X^x_{t-})\dd L_t,\quad X_0^x=x \in \RR^d_+,
\end{equation}
where $\Phi:\RR_+^d \to \RR_+^{d\times n}$ is locally Lipschitz continuous and at most of linear growth. Then there exists a unique strong solution $(X_t^x)_{t\geq 0}$ of \eqref{eq-sde}. If further $X^x_t\in \RR^d_+$ a.s. for all $t\geq 0$, then this solution has a Laplace symbol $\lambda: \RR_+^d \times \RR_+^d \to \RR$ given by
\begin{align*}
\lefteqn{\lambda(x, \xi) = \phi_L(i \Phi(x)'\xi)}\\
 & =  \ell'  \Phi(x)'\xi - \frac{1}{2} (\Phi(x)'\xi)'Q (\Phi(x)'\xi) 
 -\int_{\RR^n \backslash\{0\}} \left(\ee^{- \Phi(x)'\xi y} -1 + \Phi(x)'\xi y \cdot \mathds{1}_{\{\abs{y}<1\}}(y)\right) N(\dd y).
\end{align*}
\end{theorem}
\begin{proof}
By \cite[Cond. IX.6.7]{jacodshir} the given SDE has a unique solution. 
Since this solution is assumed to be positive (in all components) the Laplace symbol exists. The calculation of the Laplace symbol is similar to the classical one which can be found in \cite[Thm. 3.1]{sdesymbol} and is hence omitted. It heavily relies on the fact that $\Phi(y) \cdot \exp(-y)$ is bounded. \end{proof}

Sufficient conditions for non-negativity of the solution of a L\'evy driven SDE can e.g. be found in \cite[Thms. V.71 and V.72]{protter}.
In the next special case, non-negativity of the solution of the SDE is easily checked. 

\begin{corollary}
 Let $(L_t)_{t\geq 0}$ be a subordinator with Laplace exponent $\lambda_L(\xi)$, $\xi \geq 0,$ and consider the SDE \eqref{eq-sde}
where $\Phi:\RR_+^d \to \RR_+^{d\times 1}$ is locally Lipschitz continuous and at most of linear growth. 
Then there exists a unique strong solution $(X_t^x)_{t\geq 0}$, $X_t^x\in \RR^d_+$, of \eqref{eq-sde} and this solution has a Laplace symbol $\lambda: \RR_+^d \times \RR_+^d \to \RR$ given by
$\lambda(x, \xi) = \lambda_L(\Phi(x)\xi)$.
\end{corollary}

Note that this corollary as well as the theorem above (in the case governed by \cite[Thm. V.71]{protter}) yield examples of It\^o processes having differential characteristics which are  \emph{polynomially bounded} but not bounded.

We end this section by establishing the connection of the Laplace symbol and the generator of an a.s. non-negative It\^o process. To do so, we have to define the following subclass of $C^\infty_0((0,\infty)^d)$, the continuous, infinitely often differentiable functions on $(0,\infty)^d$, vanishing (componentwise) at infinity.
\begin{align} \label{eq:defS}
 \cS&:= \left\{ f\in C^\infty_0((0,\infty)^d): \int_{\RR_+} \left|(-1)^k \frac{\partial^k f}{\partial x_i^k} (x)|_{x_i= k/t} (k/t)^{k+1} \right| \dd t<\infty, \right.\\ 
& \qquad \text{ for all } k=1,2,\ldots, i=1,\ldots, d \text{ and all }  x_\ell \in \RR_+, \text{ where } \ell \neq i=1,\ldots,d,\quad \text{and} \nonumber \\ 
& \; \qquad   \lim_{\substack{j\to \infty \\ k\to \infty}} \int_{\RR_+} \left|(-1)^k \frac{\partial^k f}{\partial x_i^k} (x)|_{x_i= k/t} (k/t)^{k+1}- (-1)^j \frac{\partial^j f}{\partial x_i^j} (x)|_{x_i= j/t} (j/t)^{j+1} \right| \dd t=0 \nonumber \\ 
& \qquad  \text{ for all } i=1,\ldots, d \text{ and all }  x_\ell \in \RR_+, \text{ where } \ell \neq i=1,\ldots,d \Big\}, \nonumber 
\end{align}
where $x_i$ denotes the $i$'th component of the vector $x$.

\begin{theorem}\label{itogenerator}
 Let $(X_t)_{t\geq 0}$ be an $\RR_+^d$-valued It\^o process with polynomially bounded, finely continuous differential characteristics and with generator $(\cA, D(\cA))$. Then for all $f\in D(\cA)\cap \cS$
\begin{equation} \label{generatorsymbol}
 \cA f(x)= -\int_{\RR^d_+} \ee^{-x'\xi} \lambda(x, \xi) \check{f}(\xi) \dd \xi, \quad x\in \RR^d_+,
\end{equation}
where $\check{f}$ denotes the inverse Laplace transform of $f$.
\end{theorem}

\begin{proof}
First, observe that by a straightforward multivariate extension of \cite[Thm. VII.17a]{widder}, every $f\in\cS$ admits an integrable inverse Laplace transform $\check{f}$. Hence, starting with the definition of the Laplace symbol we obtain
\begin{align}
 \int_{\RR^d_+} \ee^{-x'\xi} \lambda(x, \xi) \check{f}(\xi) \dd \xi 
&= - \int_{\RR^d_+} \ee^{-x'\xi} \lim_{t\to 0} \frac{1}{t} \left(\EE^x[\ee^{-(X_t-x)'\xi}]-1 \right) \check{f}(\xi) \dd \xi \nonumber \\
&=  - \int_{\RR^d_+} \lim_{t\to 0} \frac{1}{t} \left(\EE^x[\ee^{-X_t'\xi}]-\ee^{-x'\xi}  \right) \check{f}(\xi) \dd \xi \nonumber  \\
&=  - \lim_{t\to 0} \frac{1}{t} \int_{\RR^d_+} \left(\EE^x[\ee^{-X_t'\xi}]-\ee^{-x'\xi}  \right) \check{f}(\xi) \dd \xi, \label{eq-generator1}
\end{align}
where in the last step we used Lebesgue's dominated convergence theorem, which is justified by Theorem \ref{glmkonv}. Further by an application of Fubini's Theorem
\begin{align*}
 \int_{\RR^d_+} \left(\EE^x[\ee^{-X_t'\xi}]-\ee^{-x'\xi}  \right) \check{f}(\xi) \dd \xi 
&= \int_{\RR^d_+}\left( \int_{\RR^d_+} \ee^{-z'\xi} \dd P_{X_t^x}(z) -  \ee^{-x'\xi}  \right) \check{f}(\xi) \dd \xi \\
&= \int_{\RR^d_+} f(z) \dd P_{X_t^x}(z) -  f(x) \\
&= \EE^x[f(X_t)] - f(x),
\end{align*}
 which, together with \eqref{eq-generator1}, yields the result.
\end{proof}


\section{Invariant Distributions}\label{S3}
\setcounter{equation}{0}


In \cite{behmeschnurr1} it was shown under certain boundedness conditions that a probability measure $\mu$ is invariant for a given It\^o process if and only if 
\begin{equation} \label{symbolstat}
 \int_{\RR^d} \ee^{ix'\xi} p(x,\xi) \mu(\dd x)=0 \quad \mbox{for all }\; \xi\in \RR^d,
\end{equation}
where $p(x,\xi)$ is the (probabilistic) symbol of the It\^o process. 
The main advantage of our new criterion is the fact, that we can relax the formerly needed boundedness conditions, namely from globally bounded by a constant to polynomially bounded.   

\begin{theorem} \label{thm:nec}
	Let $(X_t)_{t\geq 0}$ be an $\RR_+^d$-valued It\^o process with 
	polynomially bounded, finely continuous differential characteristics whose Laplace symbol is given by $\lambda(x,\xi)$, $x,\xi\in\RR_+^d$.
	\begin{enumerate}
		\item Assume that $(X_t)_{t\geq 0}$ admits an invariant distribution $\mu$, then
		\begin{equation} \label{onesidedsymbolstat}
		\int_{\RR_+^d} \ee^{-x'\xi} \lambda(x,\xi) \mu(\dd x)=0 \quad \mbox{for all }\; \xi\in \RR_+^d.
		\end{equation}
		\item Let $(\cA, D(\cA))$ be the generator of $(X_t)_{t\geq 0}$ such that the set of functions $f\in D(\cA)\cap \cS$ with $\cS$ as in \eqref{eq:defS} contains a core.
		Further assume there exists a probability measure $\mu$ on $\RR^d_+$ such that $\int_{\RR^d_+} \ee^{-x'\xi} | \lambda(x,\xi)| \mu(\dd x)<\infty$ and such that \eqref{onesidedsymbolstat} holds for all $\xi\in \RR^d_+$. Then $\mu$ is invariant for $(X_t)_{t\geq 0}$.
	\end{enumerate}
\end{theorem}

\begin{proof} To prove necessity we obtain by Lebesgue's dominated convergence theorem using Theorem \ref{glmkonv}
\begin{align*}
\int_{\RR_+^d} \ee^{-x' \xi} \lambda(x,\xi) \mu(\dd x)
&= \int_{\RR_+^d} \ee^{-x' \xi} \lim_{t\to 0} \EE^x \left[\frac{\ee^{-(X_t-x)'\xi}-1}{t}\right] \mu(\dd x)\\
&= \lim_{t\to 0}\frac{1}{t} \int_{\RR_+^d}\ee^{-x'\xi} \, \EE^x [\ee^{-(X_t-x)'\xi}-1] \mu(\dd x)
=0,
\end{align*}
as had to be shown.  For the sufficiency observe that by Theorem \ref{itogenerator} the generator $\cA$ admits the representation \eqref{generatorsymbol} with integrable Laplace inverse $\check{f}$ for all $f\in D(\cA)\cap \cS$. Therefore, for these functions $f$ we obtain using Fubini's theorem
\begin{align*}
\int_{\RR^d_+} \cA f(x)\mu(\dd x)
&= \int_{\RR^d_+} \int_{\RR_+} \ee^{-x'\xi} \lambda(x, \xi) \check{f}(\xi) \dd \xi f(x)\mu(\dd x)\\
&= \int_{\RR^d_+} \int_{\RR_+} \ee^{-x'\xi} \lambda(x, \xi) \mu(\dd x) \check{f}(\xi) \dd \xi = 0.
\end{align*}
Hence $\int_{\RR^d_+} \cA f(x)\mu(\dd x)=0,$  for all $f\in D(\cA)$
from which the assertion follows by standard arguments (e.g. proof of \cite[Prop. 9.2, b)$\Leftrightarrow$c)]{ethierkurtz}).
\end{proof}

\begin{remark}\label{reminfinv}\rm 
\alex{The condition that $D(\cA)\cap \cS$ contains a core is certainly hard to check as this is usually the case when stating conditions on the cores of a generator. Nevertheless this condition can not be weakened as $f\in \cS$ is necessary and sufficient for an integrable inverse Laplace transform of $f$ to exist (see \cite[Thm. VII.17a]{widder}).}
\end{remark}


\begin{examples} \label{ex:CBI} \begin{enumerate}
		\item  The Cox-Ingersoll-Ross (CIR) process is defined as solution of the SDE
		$$\dd X^x_t= a(b-X_t^x)\dd t + \sigma \sqrt{X_t^x} \dd B_t,\quad t\geq 0, \quad X_0^x=x>0,$$
		for some constants $a,b,\sigma>0$ and a real-valued standard Brownian motion $(B_t)_{t\geq 0}$.
		This process is a.s. non-negative and using Theorem \ref{thm:laplacesymbolausSDE} we derive its Laplace symbol
		$$\lambda(x,\xi)=a(b-x)\xi - \frac{1}{2}\sigma^2 x \xi^2 = ab\xi - (a\xi + \frac{1}{2}\sigma^2 \xi^2)x, \quad x,\xi\geq 0.$$
		Hence by Theorem \ref{thm:nec} any stationary distribution $\mu$ of the CIR process has to fulfil
		\begin{align*}
		0&= \int_{\RR_+^d} \ee^{-x'\xi} \lambda(x,\xi) \mu(\dd x)
		= ab\xi \psi_\mu (\xi) + (a\xi + \frac{1}{2}\sigma^2 \xi^2)  \psi'_\mu (\xi).
		\end{align*}
		Since $\lim_{\xi \to 0}\psi(\xi)=1$ this first order linear ODE is uniquely solved by
		$$\psi(\xi)= \left(\frac{2a}{2 a + \xi \sigma^2}\right)^{\frac{2 a b}{\sigma^2}},$$
		which is the Laplace transform of a Gamma distribution, the well-known unique stationary distribution of the CIR process.
		Note that the CIR process was not governed by the previous result \cite[Prop. 3.12]{behmeschnurr1}.
		\item 
		Consider a continuous-state branching process with immigration (CBI), i.e., a Markov process on $\RR_+$ whose Laplace symbol is given by  $\lambda (x,\xi)=F(\xi) + x G(\xi),$
		where $F$ and $G$ are of L\'evy-Khintchine form, that is,
		\begin{align*}
		F(\xi)&= a_F \xi + \int_{\RR_+} (1-\ee^{-u\xi} ) \nu_F(\dd u),  \quad \text{and}\\
		G(\xi)&= a_G \xi - \sigma_G^2 \xi^2 + \int_{\RR_+} (1-\ee^{-u\xi} - u\xi \mathds{1}_{(0,1]}(u)) \nu_G(\dd u),
		\end{align*}
		where $a_F, \sigma_G^2 \geq 0$, $a_G\in \RR$ and $\nu_F, \nu_G$ are L\'evy measures (see e.g. \cite{kellerresselmija} for details). \\
		Then for any probability measure $\mu$ with Laplace transform $\psi_\mu$ we have
		$$\int_{\RR_+} \ee^{-x\xi} \lambda (x,\xi) \mu(\dd x) = F(\xi) \psi_\mu(\xi) + G(\xi)  \psi_\mu'(\xi), \quad \xi\geq 0,$$
		and this equals $0$ if and only if 
		$$\psi_\mu(\xi)=\exp\left(\int_{(0,\xi)} \frac{F(u)}{G(u)} \dd u\right),$$
		which is the form of the Laplace transform of the invariant measure of a CBI process as shown in \cite[Thm. 2.6]{kellerresselmija}.
	\end{enumerate}
\end{examples}

\begin{remark} 
Observe that the Laplace symbol can also be used in order to analyze the invariant distributions of processes defined on the whole space $\bbr^d$ since stationarity of a Markov process is kept when applying a measurable function on the values of the process. For example in the case of a symmetric process $(X_t)_{t\geq 0}$ one can define a non-negative process $(Y_t)_{t\geq 0}$ via $Y_t:=X_t^2$ and apply Theorem \ref{thm:nec} to obtain its invariant distribution. This then also allows to determine the invariant distribution of the original process  $(X_t)_{t\geq 0}$.
\end{remark}

\section{Acknowledgements}

Our thanks go to Martin Keller-Ressel for a helpful hint leading to Example \ref{ex:CBI}(ii). Alexander Schnurr gratefully acknowledges financial support by the German Science Foundation (DFG) for the project SCHN1231/2-1. 
Further we would like to thank two anonymous referees for their comments on an earlier version of this paper which helped us to improve it.

\end{document}